\documentclass[reqno]{amsart}

\usepackage{color, float}
\usepackage{epsfig}
\usepackage{amsmath,amsfonts,amssymb,amscd,amsthm,amsbsy,bbm, epsf,calc,  pdfsync}
\usepackage{latexsym}
\usepackage{enumerate}
\usepackage[subnum]{cases}
\usepackage{appendix}

\usepackage{cite}
\usepackage{hyperref}

\newtheorem{thm}{Theorem}[section]
\newtheorem{lem}[thm]{Lemma}
\newtheorem{cor}[thm]{Corollary}
\newtheorem{prop}[thm]{Proposition}

\theoremstyle{definition}
\newtheorem{defin}[thm]{Definition}
\newtheorem{rmk}[thm]{Remark}

\DeclareMathOperator*{\diam}{diam}

\DeclareMathOperator{\spt}{spt}

\renewcommand{\epsilon}{\varepsilon}

\newcommand{\R}{\mathbb{R}}

\newcommand{\dist}{d}

\DeclareMathOperator*{\affdim}{affdim}

\newcommand{\Leb}[1]{\lvert {#1}\rvert_{\mathcal{L}}}
\newcommand{\norm}[1]{\lvert {#1}\rvert}
\def\Omegabar{\bar{\Omega}}
\def\xbar{\bar{x}}
\def\Dbar{\bar{D}}
\def\p{p}
\def\pbar{\bar{p}}

\def\mountain{m}

\def\arbitrary{{A}}

\def\sublevelset{S}

\DeclareMathOperator{\cl}{cl}

\DeclareMathOperator{\interior}{int}

\def\M{M}
\def\Mbar{\bar{M}}

\newcommand{\coord}[2]{\left[#1\right]_{#2}}

\newcommand{\targetdomcoord}[1]{\coord{\targetdom}{#1}}

\newcommand{\cExp}[2]{exp^c_{#1}({#2})}

\newcommand{\MTWcoord}[4]{-(c_{#1#2, {p}{q}}-c_{#1#2, {r}}c^{{r}, s}c_{s, {p}{q}})c^{{p}, #3}c^{{q}, #4}}
\def\subzero{\sublevelset_0}
\def\subzerocoord{\coord{\subzero}{\xbar_0}}

\newcommand{\euclidean}[2]{\langle{#1},{#2}\rangle}

\newcommand{\subdiff}[2]{\partial{#1}\left(#2\right)}
\newcommand{\subdiffbdry}[2]{\partial{#1}\left(#2\right)^{\partial}}
\newcommand{\subdiffint}[2]{\partial{#1}\left(#2\right)^{\interior}}
\newcommand{\csubdiff}[2]{\partial_c{#1}(#2)}

\newcommand{\csubdiffint}[2]{\partial_c{#1}\left(#2\right)^{\interior}}

\def\target{T}

\def\gradu{D u}

\def\targetbdry{\left(\spt{\nu}\right)^\partial}
\def\source{\spt{\mu}}
\def\target{\spt{\nu}}
\def\sourcemeas{\mu}
\def\targetmeas{\nu}
\def\sourceint{\left(\spt{\mu}\right)^{\interior}}
\def\targetint{\left(\spt{\nu}\right)^{\interior}}
\newcommand{\inner}[2]{\langle #1, #2\rangle}
\DeclareMathOperator{\domain}{dom}
\newcommand{\dom}[1]{\domain\left(#1\right)}
\newcommand{\domint}[1]{\left(\domain\left(#1\right)\right)^{\interior}}

\newcommand\nbhd[2]{\mathcal{N}_{#1}\left(#2\right)}

\def\bdry{\partial}
\def\hole{\mathcal{O}}
\def\holebdry{\hole^{\partial}}

\newcommand{\paren}[1]{\left(#1\right)}
\newcommand{\curly}[1]{\left\{#1\right\}}

\def\sourcedom{\Omega}
\def\sourcedombdry{\sourcedom^{\partial}}
\def\targetdom{\Omegabar}
\def\targetdomint{\targetdom^{\interior}}
\def\targetdombdry{\targetdom^{\partial}}
\newcommand{\csubdiffcoord}[2]{\coord{\csubdiff{#1}{#2}}{#2}}
\def\nbhd{\mathcal{N}}
\def\Breniermap{T}
\DeclareMathOperator{\dVol}{dVol}
\newcommand{\cotanspM}[1]{T^{\ast}_{#1}\M}
\newcommand{\tanspM}[1]{T_{#1}\M}
\newcommand{\cotanspMbar}[1]{T^{\ast}_{#1}\Mbar}
\newcommand{\tanspMbar}[1]{T_{#1}\Mbar}
\def\arbitrarybar{\bar A}
\def\subzeroint{\sublevelset_0^{\interior}}
\def\subzerobdry{\sublevelset_0^{\partial}}

\newcommand{\plane}[1]{\Pi^{#1}}
\def\lineseg{\ell_{\plane{\pm}}}
\newcommand{\ball}[2]{B_{#1}\paren{#2}}
\newcommand{\ccone}[1][\subzero]{K^c_{x_0, #1}}
\newcommand{\tansp}[2]{T_{#1}{#2}}

\begin{document}
\author{Young-Heon Kim}
\address{Department of Mathematics\\ University of British Columbia\\ Vancouver, V6T 1Z2 Canada}
\email{yhkim@math.ubc.ca}

\author{Jun Kitagawa}
\address{Department of Mathematics\\ University of Toronto\\
Toronto, M5S 2E4 Canada}
\email{kitagawa@math.toronto.edu}
\title{Prohibiting isolated singularities in optimal transport}
\thanks{Y.-H.Kim's research is partially supported by 
Natural Sciences and Engineering Research Council of Canada Discovery Grants 371642-09 as well as the  Alfred P. Sloan Research Fellowship 2012--2014.
Both of the authors were also supported by the Mathematical Sciences Research Institute (MSRI) as members in the thematic program ``Optimal Transport: Geometry and Dynamics".
Part of this work was done while Y.-H.Kim was visiting University Paris-Est
Cr\'eteil (UPEC) and Korea Advanced Institute of Science and
Technology (KAIST). \\
\copyright 2014 by the authors.
}

\begin{abstract}
We give natural topological conditions on the support of the target measure under which solutions to the optimal transport problem with cost function satisfying the (weak) Ma, Trudinger, and Wang condition cannot have any isolated singular points.
\end{abstract}
\maketitle

\section{Main results}\label{section: introduction} 
The optimal transport problem is the following: given \emph{source} and \emph{target} probability measure spaces $\paren{\sourcedom, \sourcemeas}$, $\paren{\targetdom, \targetmeas}$, and a measurable \emph{cost function} $c: \sourcedom\times\targetdom\to\R$, find an optimal measurable mapping $\Breniermap: \sourcedom\to\targetdom$ defined $\sourcemeas$-a.e., minimizing 
\begin{align}\label{OT}
 \int_{\sourcedom}c(x, F(x))\sourcemeas(dx)
\end{align}
over the set of all measurable $F: \sourcedom\to\targetdom$ with $F_\#\sourcemeas=\targetmeas$. 

 A fundamental problem in optimal transport theory is to understand the regularity of optimal maps. In the classical case where the cost function is given by the quadratic cost $c(x,\xbar) = \norm{x-\xbar}^2/2$ on $\R^n\times \R^n$ (or equivalently $c(x,\xbar) =- \euclidean{ x}{\xbar}$), it is well known that the optimal map is H\"older continuous  \cite{Caf90, Caf91,  Caf92} if the support of the target measure is convex, for source and target measures with densities bounded above and below. 
For more general cost functions one would require certain structural conditions, namely, ~\eqref{A1},~\eqref{A2} and especially, the $c$-convexity of the support of the target measure and the condition ~\eqref{A3w}  \cite{MTW05, TW09}  which are shown to be necessary  \cite{MTW05, Loe09} for regularity theory of the classical case to be extended (see  Section~\ref{section: convex geometry} for relevant definitions). 
Under these conditions,  H\"older continuity of the optimal map is known, under the assumption that the source and target measures have densities bounded from above and below, see \cite{Loe09, Liu09, FKM11, GK14}. For smoother measures, higher regularity theory is also known, see \cite{MTW05, TW09, LTW10}.

A natural question one can ask is what happens if one of the above structural conditions is violated. In particular, we focus on the geometric condition of convexity / $c$-convexity of the support of the target measure, where it is known that without such conditions optimal maps may not be continuous \cite{MTW05, Caf92}. 

As a first step in this direction, in this paper we analyze the case of \emph{isolated} singular points (thoughout this paper, by \emph{singularity} or \emph{singular point} we indicate a point where a function is not differentiable). In the following main theorem we prove that if the support of the target measure has no holes (by which we mean a bounded, open, connected component of the complement of the target domain), then the corresponding Brenier solution cannot have an isolated singular point in the interior of the support of the source measure. The precise statement follows with the relevant definitions given in Section~\ref{section: convex geometry}.

 \begin{thm}\label{thm: no isolated singularity in euclidean case}
Let  $\M$ and $\Mbar$  be $n$-dimensional Riemannian manifolds and $\sourcedom$ and $\targetdom$ be  bounded  domains in $\M$ and $\Mbar$, respectively. Let $c$ be a cost function  $c: \sourcedom\times\targetdom\to\R$  that satisfies \eqref{A0},~\eqref{A1},~\eqref{A2}, and~\eqref{A3w}, and assume that $\sourcedom$ and $\targetdom$ are $c$-convex with respect to each other.
 
Consider two absolutely continuous probability measures $\sourcemeas=f\dVol$ and $\targetmeas=g\dVol$ on $\M$ and $\Mbar$, respectively, with  supports $\source \subset \sourcedom$ and $\target \subset \targetdom$. Assume  that  $\target\cap\targetdombdry=\emptyset$ and 
 that there exists a constant $0<\Lambda<\infty$ such that 
\begin{align}\label{eqn: density bound}
\Lambda^{-1}\leq f,\ g\leq \Lambda
  \end{align}
on their supports. 

Finally, let $u$ be a Brenier solution (see Definition~\ref{def: brenier solution}) to the optimal transport problem with cost $c$.  For each $x_0\in\sourceint$, if there are no holes (see Definition \ref{def: holes})
 in $\target$ that are $c$-convex with respect to $x_0$, 
 then $x_0$ cannot be an isolated singular point of $u$.
\end{thm}


\begin{rmk}
 Examples of cost functions satisfying \eqref{A0}, \eqref{A1}, \eqref{A2}, and \eqref{A3w}, can be found in \cite{MTW05, TW09, Loe10, KM12}, see also \cite{Vil09}.  

 In the special case when $\M$ and $\Mbar$ are domains in Euclidean space and $c(x, \xbar) =\frac{1}{2} \norm{x-\xbar}^2$, $c$-convexity reduces to ordinary convexity (we will henceforth refer to this setting as the \emph{Euclidean case}). In the two dimensional Euclidean case, Figalli \cite{Fig10} studied the geometric structure of the singular set, and the above result on isolated singularity follows as a special case. In higher dimensions, it seems that no result on the geometric structure of singular sets (similar to the one in~\cite{Fig10})  is currently known.  For some previous related works in the Euclidean case, see~\cite{Jor55, Bey91},~\cite{Yu07}, and~\cite[Section 5]{CY09} in the case of dimension $2$, and~\cite{Bey95, FW95} in higher dimensions.
 
 While the other results mentioned above consider isolated singularities of the Monge-Amp{\`e}re equation, the papers \cite{Yu07} and~\cite[Section 5]{CY09} deal specifically with the case of the optimal transport problem  (however, still in the Euclidean case). Both results discuss the finer question of Lipschitz or $C^1$ propagation of singularities, but assume stronger conditions in addition to topological restrictions on the support of the target measure. Specifically, \cite{Yu07} assumes that all singular points have a subdifferential of affine dimension at most one, while \cite{CY09} requires the support of the source measure be convex. Our main result applies to a more general class of $c$, and also requires no hypothesis on $\source$; in fact we obtain the condition required for \cite{Yu07} in the course of our proof (see Proposition~\ref{prop: no n dim subdifferential}).
 \end{rmk}

Throughout this paper, we will denote the closure, interior, and boundary of a set $\arbitrary$ by $\arbitrary^{\cl}$, $\arbitrary^{\interior}$, and $\arbitrary^\partial$ respectively.

\section{Relevant definitions and preliminaries}\label{section: convex geometry}
In this section we gather some relevant definitions and facts about $c$-convex potential functions in relation to solutions of the optimal transport problem. Some good references are \cite{Vil09, Gut01}.

Let  $\M$ and $\Mbar$  be $n$-dimensional  Riemannian manifolds and $\sourcedom$ and $\targetdom$ be  bounded domains in $\M$ and $\Mbar$, respectively. Let $c$ be a measurable cost function  $c: \sourcedom\times\targetdom\to\R$.
We start out by stating the various assumptions we may require on our cost function $c$.

\noindent\underline{\textbf{Smoothness of cost}}:
\begin{equation}\label{A0}\tag{A0}
c\in C^4(\sourcedom^{\cl}\times \targetdom^{\cl}).
\end{equation}

\noindent\underline{\textbf{Twist}}:

We will say $c$ satisfies condition~\eqref{A1} if each of the mappings
\begin{equation}
\begin{array}{rl}
\targetdom\ni\xbar & \mapsto -Dc(x_0, \xbar)\in\cotanspM{x_0},\\
\sourcedom\ni x & \mapsto -\Dbar c(x, \xbar_0)\in\cotanspMbar{\xbar_0},
\end{array} \label{A1}\tag{Twist}
\end{equation}
are injective for each $x_0 \in\sourcedom$ and $\xbar_0 \in \targetdom$.
Here, $D$, $\bar D$ denote the usual differential in the $x$ or $\xbar$ variable. 

\begin{rmk}\label{rmk: twist comments}
We use the standard notation $\cExp{x_0}{\cdot}$ and $\cExp{\xbar_0}{\cdot}$ to denote the inverses of the above two mappings. Also, for any $\arbitrary\subset \sourcedom$, $\xbar\in\targetdom$ or $\arbitrarybar\subset\targetdom$, $x\in\sourcedom$, we will write
\begin{align*}
 \coord{\arbitrary}{\xbar}:&=-\Dbar c(\arbitrary, \xbar),\\
 \coord{\arbitrarybar}{x}:&=-Dc(x, \arbitrarybar).
\end{align*}
We also comment here, for the cost $c(x, \xbar) = -\euclidean{x}{\xbar}$ on $\R^n\times \R^n$, these mappings are both just the identity map.

\begin{defin}[\noindent\underline{\textbf{$c$-convexity of a set}}\cite{MTW05}]\label{defin: c-convex sets} If $\arbitrary\subset \sourcedom$ and $\xbar\in\targetdom$, we say that \emph{$\arbitrary$ is $c$-convex with respect to $\xbar$} if the set $\coord{\arbitrary}{\xbar}$ is a convex subset of $\cotanspMbar{\xbar}$. If $\arbitrarybar\subset\targetdom$ and $x\in\sourcedom$, we define when \emph{$\arbitrarybar$ is $c$-convex with respect to $x$} and \emph{$\arbitrary$ and $\arbitrarybar$ are $c$-convex with respect to each other} in the obvious way.
\end{defin}
\end{rmk}
\noindent\underline{\textbf{Nondegeneracy}}:

We say $c$ satisfies condition~\eqref{A2} if, for each $x\in\sourcedom$ and $\xbar\in\targetdom$, the linear mapping
\begin{equation}
\begin{array}{rl}
-\Dbar D c(x, \xbar): \tanspMbar{\xbar} \to \tansp{-Dc(x, \xbar)}{\paren{\cotanspM{x}}\cong \cotanspM{x}}
\end{array}\label{A2}\tag{Nondeg}
\end{equation}
is invertible
(and consequently, so is its adjoint mapping, $-D\Dbar c(x, \xbar):  \tanspM{x}\to \cotanspMbar{\xbar}$).

\noindent\underline{\textbf{MTW (Ma-Trudinger-Wang condition)}} \cite{MTW05, TW09, Loe09}:
We say $c$ satisfies the condition~\eqref{A3w} if, for any $x\in\sourcedom$, $\xbar\in\targetdom$, and $V\in \tanspM{x}$, $\eta\in \cotanspM{x}$ with  $\eta(V)=0$,
\begin{equation}\label{A3w}\tag{MTW}
\MTWcoord{i}{j}{k}{l}(x, \xbar)V^iV^j\eta_k\eta_l\geq 0.
\end{equation}
Here we fix coordinate systems on $\M$ and $\Mbar$ and take all derivatives with respect to these coordinates; lower indices before a comma denote derivatives of $c$ with respect to the $x$ variable, and lower indices after a comma denote derivatives with respect to the $\xbar$ variable. Also, a pair of raised indices denotes the inverse of a matrix.

We next define some basic concepts of use in $c$-convex geometry.
\begin{defin}\label{def: c-convex functions}
A real valued function $u$ defined on $\sourcedom$ is said to be \emph{$c$-convex} if for any $x_0\in\sourcedom$, there exists some $\xbar_0\in\targetdom$ and $\lambda_0\in\R$ such that
 \begin{align*}
 -c(x_0, \xbar_0)+\lambda_0&=u(x_0),\\
 -c(x, \xbar_0)+\lambda_0&\leq u(x)
\end{align*}
for all $x\in\sourcedom$. 
Any function of the form $-c(\cdot, \xbar_0)+\lambda_0$ is called a \emph{$c$-affine function (with focus $\xbar_0$)}, and if it satisfies the above relations is said to \emph{support $u$ from below at $x_0$}.
\end{defin}

We also define the $c$-subdifferential of a $c$-convex function, and the subdifferential of a semi-convex function.
\begin{defin}\label{def: subdifferential}
The \emph{subdifferential} of a semi-convex function $u$ at a point  $x\in\domint{u}$ is defined by the set
\begin{align*}
 \subdiff{u}{x}:=\curly{\pbar\in\cotanspM{x}\mid u(x)+\inner{v}{\pbar}+o(\norm{v})\leq u(\exp_{x}{(v)}),\ \tanspM{x}\ni v\to 0},
\end{align*}
here $\exp_{x}$ is the Riemannian exponential mapping on $\M$.

Similarly, the \emph{$c$-subdifferential} of a $c$-convex function $u$ at a point  $x\in\domint{u}$ is defined as the set
\begin{align*}
 \csubdiff{u}{x}:=\curly{\xbar\in\targetdom\mid -c(y, \xbar)+c(x, \xbar)+u(x)\leq u(y),\ \forall y\in\dom{u}}.
\end{align*}
If $\arbitrary\subseteq\sourcedom$, we write
\begin{align*}
 \csubdiff{u}{\arbitrary}:=\bigcup_{x\in \arbitrary}\csubdiff{u}{x}.
\end{align*}
\end{defin}
\begin{rmk}\label{rmk: c-subdifferential remarks}
Note that if $u$ is semi-convex, each $\subdiff{u}{x} $ is a nonempty, convex set, and for any point $x$ where $u$ is differentiable, we have $\subdiff{u}{x}=\curly{\gradu(x)}$. Additionally, it is known that if $c$ satisfies~\eqref{A0}, then a $c$-convex function is semi-convex, hence in particular it is differentiable a.e.

Additionally, if $u$ is $c$-convex it is not difficult to see that its $c$-subdifferential is \emph{$c$-monotone}, i.e. for any $x_0$, $x_1\in\sourcedom$ and $\xbar_0\in\csubdiff{u}{x_0}$, $\xbar_1\in\csubdiff{u}{x_1}$, we have
\begin{align*}
 c(x_0, \xbar_0)+c(x_1, \xbar_1)\leq c(x_0, \xbar_1)+c(x_1, \xbar_0).
\end{align*}


\end{rmk}
\begin{defin}\label{def: brenier solution}
Suppose $c$ satisfies~\eqref{A1}. A \emph{Brenier solution (to the optimal transport problem with cost $c({x},{\xbar})$) pushing $\sourcemeas$ forward to $\targetmeas$} is a $c$-convex function $u$ defined on $\source$ such that 
\begin{align*}
 \Breniermap_\#\sourcemeas&=\targetmeas,\\
 \Breniermap\paren{\source}&\subseteq\target,
\end{align*}
where $\Breniermap$ is the \emph{Brenier map} defined for a.e. $x$ (where $u$ is differentiable) by
\begin{align*}
 \Breniermap(x):&=\cExp{x}{\gradu(x)}.
\end{align*}
\end{defin}

If $u$ is a Brenier solution pushing $\sourcemeas$ forward to $\targetmeas$, then it is well known that $\Breniermap$ as defined above is optimal in~\eqref{OT}.

The following result (discovered by Loeper~\cite{Loe09} for domains in $\R^n$, further developed in~\cite{TW09a,KM10, LV10, FRV11}, and extended to domains in manifolds under certain conditions) details certain geometric properties of $c$-convex functions. It will play a key role in our main proof.
\begin{thm}[Loeper's maximum principle~\cite{Loe09}]\label{thm: DASM}
Suppose $c$, $\sourcedom$, and $\targetdom$ satisfy the conditions of Theorem~\ref{thm: no isolated singularity in euclidean case}.
 Also let $x_0\in\sourcedom$, $\pbar_0$, $\pbar_1\in\targetdomcoord{x_0}$, and $\xbar(t):=\cExp{x_0}{(1-t)\pbar_0+t\pbar_1}$.  Then for any $x\in\sourcedom$,
\begin{align}\label{eq:LoeperMP}
&-c(x, \xbar(t))+c(x_0, \xbar(t))\\\nonumber
&\qquad\leq \max \curly{-c(x, \xbar(0)) +c(x_0,  \xbar(0)), -c(x, \xbar(1))+c(x_0, \xbar(1))}.
\end{align}
An analogous inequality holds with the roles of $\sourcedom$ and $\targetdom$ reversed.
\end{thm}

This lemma has several important consequences, we will require the following two of them later; the second of which was first observed and used in \cite{FKM09, FKM11} and \cite{Liu09}.
\begin{cor}[{\cite[Theorem 3.1]{Loe09}}]\label{cor: c-convex c-subdifferential}
Suppose $c$, $\sourcedom$, and $\targetdom$ satisfy the same conditions as Theorem~\ref{thm: DASM} above, and $u$ is a $c$-convex function on $\sourcedom$.
Then  for any $x_0\in\sourcedom$, 
 \begin{align}\label{eq:c-sub is sub}
\csubdiffcoord{u}{x_0}=\subdiff{u}{x_0},
 \end{align}
 in particular,  $\csubdiff{u}{x_0}$ is $c$-convex with respect to $x_0$.
\end{cor}
\begin{cor}\label{cor: c-convex sublevelset}
Suppose $c$, $\sourcedom$, and $\targetdom$ satisfy the same conditions as Theorem~\ref{thm: DASM}, and $u$ is a $c$-convex function on $\sourcedom$.
Then, for any $\xbar_0\in\Omegabar$ and $\lambda_0\in\R$, the {\em section} 
\begin{align*}
 \{x\in\sourcedom\mid u(x)\leq -c(x, \xbar_0)+\lambda_0\}
\end{align*}
is $c$-convex with respect to $\xbar_0$.
\end{cor}

We also state here a fairly standard result concerning $c$-subdifferentials of \emph{$c$-cones}.
\begin{lem}\label{lem: c-cones}
 Suppose $c$, $\sourcedom$, and $\targetdom$ satisfy the conditions of Theorem~\ref{thm: no isolated singularity in euclidean case},
 $u$ is a $c$-convex function, $\mountain_0$ is a $c$-affine function with focus $\xbar_0$, and let $\subzero:=\curly{u\leq \mountain_0}$ be such that $\subzero\cap\sourcedombdry=\emptyset$. Fix $x_0\in\subzeroint$ and define \emph{the $c$-cone over the section $\subzero$ with vertex $x_0$} by
\begin{align*}
 \ccone(x):=\sup_{\mountain}\mountain(x),
\end{align*}
 where the supremum is taken over all $c$-affine functions $\mountain$ satisfying $\mountain\leq \mountain_0$ on $\subzerobdry$, and $\mountain(x_0)\leq u(x_0)$. Then, 
\begin{align}\label{eqn: c-cone inclusions}
\csubdiff{\ccone}{x_0}\subset\csubdiff{u}{\subzero},
\end{align}
and if $\xbar_0\in\targetdomint$, 
\begin{align}\label{eqn: interior of c-cone subdiff}
 -Dc(x_0, \xbar_0)\in\csubdiffcoord{\ccone}{x_0}^{\interior}.
\end{align}
\end{lem}
\begin{proof}
 A proof of~\eqref{eqn: c-cone inclusions} 
 is contained in, for example,~\cite[Lemma 3.4]{GK14}.

We will show~\eqref{eqn: interior of c-cone subdiff}. By assumption, $\mountain_0(x_0)-u(x_0)>0$. Let us write $\pbar_0:=-Dc(x_0, \xbar_0)$, then recall that $\cExp{x_0}{\pbar_0}=\xbar_0$. Hence for a sufficiently small $r_0>0$, we have  (for some $C>0$ depending only on derivatives of the cost $c$) that  for all  $\pbar\in\ball{r_0}{\pbar_0}$, the function 
$\mountain_{\pbar}(x):=-c(x, \cExp{x_0}{\pbar})+c(x_0, \cExp{x_0}{\pbar})+u(x_0)$ satisfies
\begin{align*}
\mountain_{\pbar}(x)
& =(-c(x, \cExp{x_0}{\pbar})+c(x_0, \cExp{x_0}{\pbar})+\mountain_0(x_0))-\mountain_0(x_0)+u(x_0)\\
&
\leq  \mountain_0(x)+Cr_0-(\mountain_0(x_0)-u(x_0))\\
 &
 <\mountain_0(x)
 \end{align*}
for all $x\in\subzerobdry$. Thus $\mountain_{\pbar}$ is admissible in the supremum defining $\ccone$, and it must support the $c$-cone $\ccone$ from below at $x_0$. In particular we have for all $\pbar\in\ball{r_0}{\pbar_0}$ that $\cExp{x_0}{\pbar}\in \csubdiff{\ccone}{x_0}$, hence by Corollary~\ref{cor: c-convex c-subdifferential}, $\pbar\in\csubdiffcoord{\ccone}{x_0}$, proving~\eqref{eqn: interior of c-cone subdiff}.
\end{proof}
%
%
%
Finally, we give the precise definition of  a {\em hole}.
\begin{defin}\label{def: holes}
Given any domain $\arbitrary$, we say that $\hole$ is a \emph{hole in $\arbitrary$} if $\hole\neq\emptyset$ is a bounded, open, connected set such that
\begin{align*}
 \hole\cap \arbitrary^{\interior}&=\emptyset,\\
 \holebdry&\subset \arbitrary^{\partial}.
\end{align*}
\end{defin}

\section{Proof of Theorem~\ref{thm: no isolated singularity in euclidean case}
}\label{section: main proof}
We begin by deriving several intermediate results. We start with stating a very useful tool in our analysis, due to Albano and Cannarsa:
\begin{prop}[{\cite[Theorem 4.2]{AC99}}]\label{thm: singularity propagation}
 Suppose that $u$ is a semi-convex function and $x_0\in\domint{u}$ is a point where $u$ is not differentiable. If there exists an open neighborhood $\nbhd$ of $x_0$ such that $u$ is differentiable on $\nbhd\setminus\curly{x_0}$, 
 then for every $p\in\subdiffbdry{u}{x_0}$ there exists a sequence $x_k\to x_0$ such that $\gradu(x_k)\to p$ as $k\to \infty$.
\end{prop}

The next result excludes having a full dimensional subdifferential at an isolated singular point, when the target domain contains no holes. Note that the result can be shown under just the condition~\eqref{A1}, and can be strengthened under ~\eqref{A2} and~\eqref{A3w}. We also comment that this will be the only place where we use the no-hole condition on $\target$, for the  proof of Theorem~\ref{thm: no isolated singularity in euclidean case}.

\begin{prop}\label{prop: no n dim subdifferential}
Suppose that $c$ is $C^1$ and satisfies~\eqref{A1}, $u$ is a $c$-convex Brenier solution, and $\target$ contains no holes. Then $u$ cannot have any isolated singular point $x_0\in\sourceint$ with $\affdim{\subdiff{u}{x_0}}=n$ (here $\affdim$ is the affine dimension of a convex set).

If in addition, $c$ satisfies~\eqref{A0},~\eqref{A2}, and~\eqref{A3w}, and $\sourcedom$ and $\targetdom$ are $c$-convex with respect to each other, we obtain the same conclusion under the weaker condition that $\target$ contains no holes $c$-convex with respect to $x_0$.
\end{prop}
\begin{proof}
Suppose by contradiction that $x_0\in\sourceint$ is an isolated singular point of $u$, and the affine dimension of $\subdiff{u}{x_0}$ is $n$. Since $c$ is $C^1$ and satisfies~\eqref{A1}, the mapping $\cExp{x_0}{\cdot}$ is continuous and injective,
thus Brouwer's invariance of domain theorem (see~\cite{Bro71}) gives that $\cExp{x_0}{\cdot}$ is a homeomorphism between the open set $\subdiffint{u}{x_0}$ and its image. In particular, $\cExp{x_0}{\subdiffint{u}{x_0}}$ is a nonempty, open, bounded, connected set. Then since $x_0$ is an isolated singularity, by Proposition~\ref{thm: singularity propagation} we have 
\begin{align}\label{eqn: boundary of subdifferential}
\cExp{x_0}{\subdiffbdry{u}{x_0}}\subset \target\cap \csubdiff{u}{x_0},
\end{align}
 as $\gradu (\dom {\gradu}) \subset \target$ for the  Brenier solution $u$.

We now claim that 
\begin{align}\label{eqn: no other points in interior subdifferential}
\cExp{x_0}{\subdiffint{u}{x_0}}\cap \csubdiff{u}{x_1}=\emptyset
\end{align}
for any $x_1\neq x_0$. First, fix such an $x_1\in\sourcedom$ and define
\begin{align*}
 F(\xbar):=c(x_0, \xbar)-c(x_1, \xbar),
\end{align*}
which is a $C^1$ function satisfying $DF(\xbar)\neq 0$ for any $\xbar$ (by~\eqref{A1}). In particular, $F$ cannot attain its maximum over the compact set $\cExp{x_0}{\subdiff{u}{x_0}}$ except at the boundary, say at $\xbar_0\in\cExp{x_0}{\subdiffbdry{u}{x_0}}\subset \csubdiff{u}{x_0}$. Thus if there exists $\xbar_1\in \cExp{x_0}{\subdiffint{u}{x_0}}\cap \csubdiff{u}{x_1}$, this would imply that
\begin{align*}
 F(\xbar_1)&<F(\xbar_0)\\
 \iff c(x_0, \xbar_1)-c(x_1, \xbar_1)&<c(x_0, \xbar_0)-c(x_1, \xbar_0)\\
 \iff c(x_0, \xbar_1)+c(x_1, \xbar_0)&<c(x_0, \xbar_0)+c(x_1, \xbar_1),
\end{align*} 
which is a violation of $c$-monotonicity of the $c$-subdifferential of $u$ (see Remark~\ref{rmk: c-subdifferential remarks}). As a result there cannot be such an $\xbar_1$, and we obtain~\eqref{eqn: no other points in interior subdifferential}.
%
%
 Since $\targetmeas=\Breniermap_{\#}\sourcemeas$, we must then have $$\cExp{x_0}{\subdiffint{u}{x_0}}\cap\target=\emptyset.$$ However, when combined with~\eqref{eqn: boundary of subdifferential} this exactly implies that $\cExp{x_0}{\subdiffint{u}{x_0}}$ is a hole in $\target$ which contradicts our initial assumption, therefore it must be that $\affdim{\subdiff{u}{x_0}}<n$.

If $c$ also satisfies~\eqref{A0},~\eqref{A2}, and~\eqref{A3w}, by Corollary~\ref{cor: c-convex c-subdifferential} we have that $\cExp{x_0}{\subdiffint{u}{x_0}}=\csubdiffint{u}{x_0}$ and is $c$-convex with respect to $x_0$; the conclusion thus follows from the same proof as above.
\end{proof}

In the next lemma, we extend to $c$-convex functions the following easy result about convex functions: if a convex function $u$ makes contact with an affine function along a line segment containing a point $x_0$, then either $u$ is singular along the line segment or the gradient of the affine function is an exposed point of the convex set $\subdiff{u}{x_0}$. Our extension is, in particular, to cost functions such that Loeper's maximum principle, Theorem \ref{thm: DASM} holds.

\begin{lem}\label{lem: contact set is singleton with MTW}
Suppose that $c$ satisfies~\eqref{A0},~\eqref{A1},~\eqref{A2}, and~\eqref{A3w} (so that Loeper's maximum principle, Theorem \ref{thm: DASM} \eqref{eq:LoeperMP} and its consequences, Corollary \ref{cor: c-convex c-subdifferential} \eqref{eq:c-sub is sub} and Corollary~\ref{cor: c-convex sublevelset} hold). 
Also let 
$u$ be a $c$-convex function on $\sourcedom$ 
and assume that $x_0$ is an isolated singular point of $u$. 
Then if $\pbar_0$ is not an extremal point of the convex set  $\coord{\csubdiff{u}{x_0}}{x_0}$ 
and $\xbar_0:=\cExp{x_0}{\pbar_0}$, the contact set
\begin{align*}
 \subzero:=\curly{x\in\sourcedom\mid u(x)=-c(x, \xbar_0)+c(x_0, \xbar_0)+u(x_0)}
\end{align*}
consists only of the single point $x_0$.
\end{lem}
\begin{proof}
 Fix a  $\pbar_0$ that is not an extremal point of $\coord{\csubdiff{u}{x_0}}{x_0}$.  There exist $\pbar_{\pm}\neq \pbar_0$ such that $\pbar_{\pm}\in\coord{\csubdiff{u}{x_0}}{x_0}$ and $\pbar_0=\frac{1}{2}(\pbar_++\pbar_-)$; let us write $\xbar_{\pm}:=\cExp{x_0}{\pbar_{\pm}}$.

Now, suppose by contradiction that there exists some $x_1\in\subzero$ with $x_1\neq x_0$.  
Consider the  $c$-segment $x(\lambda):=\cExp{\xbar_0}{(1-\lambda)\p_0+\lambda\p_1}$, for $\lambda\in[0, 1]$ from $x_0$ to $x_1$; observe from Corollary~\ref{cor: c-convex sublevelset} that $x(\lambda) \in \subzero$ for all $\lambda \in [0,1]$. 
Also using that $\xbar_\pm\in\csubdiff{u}{x_0}$,   
we must have 
\begin{align}\label{eq: subdiff x pm}
 \max\curly{-c(x, \xbar_+)+c(x_0, \xbar_+),\ -c(x, \xbar_-)+c(x_0, \xbar_-)}+u(x_0) \le u(x)
\end{align}
for all $x \in \sourcedom$.
In particular, 
\begin{align*}
 -c(x(\lambda), \xbar_\pm)+c(x_0, \xbar_\pm)+u(x_0)
 &\leq u(x(\lambda))\\
 & = - c(x(\lambda), \xbar_0)+c(x_0, \xbar_0)+u(x_0)
\end{align*}
for all $\lambda\in[0, 1]$. At the same time by using Theorem~\ref{thm: DASM} \eqref{eq:LoeperMP},
\begin{align*}
 -c(x(\lambda), \xbar_0)+c(x_0, \xbar_0)&\leq \max\curly{-c(x(\lambda), \xbar_+)+c(x_0, \xbar_+),\ -c(x(\lambda), \xbar_-)+c(x_0, \xbar_-)},
\end{align*}
thus by combining these we must have the equality
\begin{align*}
 \max\curly{-c(x(\lambda), \xbar_+)+c(x_0, \xbar_+),\ -c(x(\lambda), \xbar_-)+c(x_0, \xbar_-)}+u(x_0) &=u(x(\lambda))
\end{align*}
for all $\lambda \in[0, 1]$. 
 Together with \eqref{eq: subdiff x pm},  this implies that for each $\lambda \in [0,1]$, either $\xbar_+ \in \csubdiff{u}{x(\lambda)}$ or $\xbar_- \in \csubdiff{u}{x(\lambda)}$. Since $\xbar_+$, $\xbar_-\ne \xbar_0$ by construction, and clearly $\xbar_0\in \csubdiff{u}{x(\lambda)}$ for all $\lambda\in[0, 1]$ this implies all points $x(\lambda)$ in the $c$-segment must be singular points, contradicting that $x_0$ is an isolated singular point. This proves $\subzero = \curly{x_0}$.
\end{proof}

In order to prove the main theorem, we require a modified version of the estimate~\cite[Lemma 6.10]{FKM11} 
(this is proven in the same vein as~\cite[Proposition 1]{FK10} for the Euclidean case of $c(x, \xbar)=-\euclidean{x}{\xbar}$). By the notation $\Leb{\cdot}$, we denote the volume of a set in $\M$, $\Mbar$ or an associated cotangent space, induced by the Riemannian metric on either $\M$ or $\Mbar$ (which will be clear from context).

\begin{lem}\label{lem: modified aleksandrov}
 Suppose $c$, $u$, $\sourcedom$, $\targetdom$, $\sourcemeas$, and $\targetmeas$ satisfy the conditions of Theorem~\ref{thm: no isolated singularity in euclidean case}.
%
  Also let $\mountain_0$ be a $c$-affine function with focus $\xbar_0$, let $\subzero:=\curly{u\leq \mountain_0}$ with $\subzero\cap\sourcedombdry=\emptyset$, fix two parallel planes $\plane{+}$ and $\plane{-}$ in $\cotanspMbar{\xbar_0}$ supporting the (convex) set $\subzerocoord$ from opposite sides, and let $\lineseg$ be the length of the longest line segment orthogonal to $\plane{\pm}$ that is contained in $\subzerocoord$. Finally, suppose for some $\delta>0$, $x_0\in\subzeroint$ is such that there exists $\pbar_\delta\in \coord{\csubdiff{u}{x_0}\cap \targetint}{x_0}$ with $\dist{\paren{\pbar_\delta, \coord{\targetbdry}{x_0}}}\geq\delta$. Then (writing $p_0:=-\Dbar c(x_0, \xbar_0)$) there exists a constant $C>0$ depending only on $\delta$, $n$, $\Lambda$, $\diam{\paren{\target}}$, and $c$ such that
\begin{align*}
 \paren{\mountain_0(x_0)-u(x_0)}^n\leq\frac{C\min{\curly{\dist{\paren{p_0, \plane{+}}}, \dist{\paren{p_0, \plane{-}}}}}}{\lineseg}\Leb{\subzero}^2
\end{align*}
\end{lem}
\begin{proof}
First, one can use~\eqref{eqn: density bound} and follow a proof analogous to~\cite[Lemma 3.4]{FK10} (using Remark~\ref{rmk: c-subdifferential remarks}, and replacing the Legendre transform of a function by the \emph{$c$-transform}, see~\cite[Section 3]{FKM11}), to obtain 
\begin{align*}
\Leb{ \coord{\csubdiff{u}{\subzero}}{x_0}\cap\coord{\target}{x_0}}= C\Leb{\csubdiff{u}{\subzero}\cap\target}\leq \Lambda^2C \Leb{\subzero},
\end{align*}
where $C>0$ depends on the cost function $c$. Now let $\ccone (\cdot)$ be the $c$-cone over the section $\subzero$ with vertex $x_0$. Then, by using~\cite[Lemma 3.1]{FK10}, we calculate
\begin{align*}
 \Leb{\csubdiffcoord{\ccone}{x_0}}&\leq C(\delta, \diam{\paren{\target}})\Leb{\csubdiffcoord{\ccone}{x_0}\cap\ball{\delta}{\pbar_\delta}}\\
 &\leq C(\delta, \diam{\paren{\target}})\Leb{ \coord{\csubdiff{u}{\subzero}}{x_0}\cap\coord{\target}{x_0}}\\
 &\leq C\Leb{\subzero},
\end{align*}
where the final constant $C$ depends on $c$, $\Lambda$, $\delta$, and $\diam{\paren{\target}}$. 
Combining this with the original proof of~\cite[Lemma 6.10]{FKM11}, we immediately obtain the claim.
\end{proof}

With all of the preceeding ingredients in hand, we are ready to prove the main theorem.
\begin{proof}[\bf { Proof of Theorem~\ref{thm: no isolated singularity in euclidean case}}]
Suppose by contradiction that $u$ has an isolated singular point $x_0\in\sourceint$.

 We begin by a localization of $u$ around $x_0$.  $\csubdiffcoord{u}{x_0}$ is convex by Corollary \ref{cor: c-convex c-subdifferential} \eqref{eq:c-sub is sub} and contains more than one point since $u$ is singular at $x_0$; thus there must exist at least one non-extremal point $\pbar_0$ of $\csubdiffcoord{u}{x_0}$. 
Let us define a family of sections around $x_0$ using $c$-affine functions with  focus $\xbar_0:= \cExp{x}{\pbar_0}$, for $h>0$ let 
\begin{align*}
 \sublevelset_h:=\curly{x\in\sourcedom\mid u(x)\leq -c(x, \xbar_0)+c(x_0, \xbar_0)+u(x_0)+h}.
\end{align*}
Notice that by Lemma~\ref{lem: contact set is singleton with MTW}, 
it holds the section is a singleton when $h=0$, i.e.  $\subzero=\curly{x_0}$.  
As a result $S_h$ can be made sufficiently small around $x_0$ for small enough  $h>0$. Thus by the assumption that $x_0$ is an isolated singularity, we may assume $h>0$ to be small enough that $\sublevelset_h\subset\sourceint$ and  $u$ is differentiable on $\sublevelset_h\setminus\curly{x_0}$. 

On the other hand, by Proposition~\ref{prop: no n dim subdifferential} we see that the affine dimension of $\subdiff{u}{x_0}$ is strictly less than $n$, and in particular $\subdiff{u}{x_0}=\subdiffbdry{u}{x_0}$. Hence by  Proposition~\ref{thm: singularity propagation}, the definition of Brenier solution, and closedness of $\target$, we see that 
\begin{align}\label{eqn: subdifferential contained in closure}
 \csubdiff{u}{x_0}=\cExp{x_0}{\subdiff{u}{x_0}}\subset \target.
\end{align}
In particular, $\xbar_0 \in \target$. Since  $u$ is differentiable on $\sublevelset_h\setminus\curly{x_0}$, \eqref{eqn: subdifferential contained in closure} 
and the definition of Brenier solution imply 
 \begin{align}\label{eq: inside target}
 \csubdiff{u}{\sublevelset_h} \subset \target .
\end{align}

 Now consider the $c$-cone  $\ccone[\sublevelset_h](x)$ over  $\sublevelset_h$ with vertex $x_0$ as in Lemma~\ref{lem: c-cones}.   From the condition $\target\cap\targetdombdry=\emptyset$, it holds 
$\xbar_0\in\targetdomint$, therefore we can
 apply Lemma~\ref{lem: c-cones}~\eqref{eqn: c-cone inclusions}
and~\eqref{eqn: interior of c-cone subdiff} to see
\begin{align*}
-Dc(x_0, \xbar_0)&\in \csubdiffcoord{\ccone[\sublevelset_h]}{x_0}^{\interior}\subset \coord{\csubdiff{u}{\sublevelset_h}}{x_0}.
\end{align*}
From \eqref{eq: inside target}, this implies $-Dc(x_0, \xbar_0)\in\coord{\target}{x_0}^{\interior}$.


However if this is the case, then one can follow the proof of~\cite[Theorem 8.3]{FKM11}, using Lemma~\ref{lem: modified aleksandrov} above (with $\delta=\dist{\paren{-Dc(x_0, \xbar_0), \coord{\target}{x_0}^{\bdry}}}>0$) in place of~\cite[Theorem 6.11]{FKM11}, to obtain that $u$ is differentiable at $x_0$; this contradicts that $x_0$ is a singular point, completing the proof.
 \end{proof}

\bibliography{mybiblioinjective}
\bibliographystyle{plain}
\end{document}